\documentclass[10pt,twoside]{amsart}
\usepackage{pstricks}
\usepackage[bookmarks=true]{hyperref}

\newtheorem{thm}{Theorem}[section]

\newtheorem{cor}[thm]{Corollary}
\newtheorem{lem}[thm]{Lemma}
\theoremstyle{definition}
\newtheorem{defn}[thm]{Definition}

\newtheorem{rem}[thm]{Remark}

\newcommand{\hgt}{{\rm ht}\,}
\newcommand{\bight}{{\rm bight}\,}
\newcommand{\ara}{{\rm ara}\,}
\newcommand{\pd}{{\rm pd}\,}

\title[]{The arithmetical rank of the edge ideals \\ of graphs with pairwise disjoint cycles}

\author{Margherita Barile}
\address[M. Barile]{Dipartimento di Matematica, Universit\`a degli Studi di Bari ``Aldo Moro'', Via Orabona 4, 70125 Bari, Italy}
\email{margherita.barile@uniba.it}

\author{Antonio Macchia}
\address[A. Macchia]{Fachbereich Mathematik und Informatik, Philipps-Universit\"at Marburg, Hans-Meerwein-Strasse 6, 35032 Marburg, Germany}
\email{macchia.antonello@gmail.com}

\begin{document}

\begin{abstract}
We prove that, for the edge ideal of a graph whose cycles are pairwise vertex-disjoint, the arithmetical rank is bounded above by the sum of the number of cycles and the maximum height of its associated primes.
\end{abstract}

\maketitle



\section{Introduction}

Let $R$ be a polynomial ring over a field. Any ideal $I$ of $R$ generated by squarefree quadratic monomials can be viewed as the so-called \textit{edge ideal} $I(G)$ of a graph $G$ whose vertex set $V(G)$ is the set of indeterminates, and whose edges are the sets formed by two indeterminates $x,y$ such that $xy$ is a generator of $I$. This notion was introduced in 1990 by Villarreal \cite{V90} and extensively studied in 1994 by Simis, Vasconcelos and Villarreal \cite{SVV}; see \cite{MV11} for a detailed survey. The present paper is concerned with two algebraic invariants of $I(G)$: the \textit{big height}, denoted by $\bight I(G)$, which is the maximum height of the minimal primes of $I(G)$, and the \textit{arithmetical rank}, denoted by $\ara I(G)$, which is the minimum number of elements of $R$ that generate an ideal whose radical is $I(G)$. It is well known that
\begin{equation}\label{0}
\hgt I(G)\leq\bight I(G)\leq \pd R/ I(G)\leq \ara I(G),
\end{equation}
where $\hgt$ and $\pd$ denote the height and the projective dimension, respectively. A very special case is the one where equality holds everywhere: then the ideal $I(G)$ is a set-theoretic complete intersection. According to some recent results, this occurs for many Cohen-Macaulay edge ideals fulfilling additional conditions like,  e.g., having height two \cite{K} or having height equal to half the number of vertices \cite{BM}. These include the bipartite graphs studied in \cite{HH05} and in \cite{EV97}.    A more general case is the one where the arithmetical rank is equal to the projective dimension of the quotient ring. This equality has been proven for several classes of graphs, such as the graphs formed  by one cycle or by two cycles having one vertex in common \cite{BKMY12} or connected through an edge \cite{M},  those formed by some cycles and lines having a common vertex \cite{KM12}, or those whose edge ideals are subject to certain algebraic constraints (see, e.g., \cite{EOT10} and \cite{KRT}).
A stronger condition is the equality between the arithmetical rank and the big height, which has been established for certain unmixed bipartite graphs \cite{Ku09}, for acyclic graphs \cite{KT13}, for graphs formed by a single cycle and some terminal edges attached to some of its vertices (\textit{whisker graphs on a cycle}) (see \cite{M13a} or \cite{M13c}), and  for graphs in which every vertex  belongs to a terminal edge (\textit{fully whiskered graphs}) (see \cite{M13b} or \cite{M13c}). A question that naturally arises when comparing the  arithmetical rank and the big height, is whether their difference can be bounded above by means of some graph-theoretical invariants. We will show that, for every graph whose cycles are pairwise disjoint, an upper bound is provided by the number of cycles. This is a generalization of the result in \cite{KT13}, but is proven independently, and by completely different techniques. The approach is inductive on the number of edges, and the basis of induction is the case of fully whiskered graphs, for which the claim was proven by the second author using the homological method, based on Lyubeznik resolutions, introduced by Kimura in \cite{Ki09}. All the results proven in this paper hold on any field.

\section{Preliminaries}

We first introduce some graph-theoretical terminology and notation.\\
All graphs considered in this paper are simple, i.e., without multiple edges or loops. Given two vertices $x$ and $y$ of a graph $G$, we will say that $x$ is a \textit{neighbour} of $y$ if the vertices $x,y$ form an edge. By abuse of notation, this edge will be denoted by $xy$, with the same symbol used for the corresponding monomial of $I(G)$. The vertex $x$ will be called \textit{terminal} or a \textit{leaf}, if it has exactly one neighbour $y$; in this case the edge  $xy$ will be called \textit{terminal}. For the remaining basic terminology about graphs we refer to \cite{Ha}.\\
A graph will always be identified with the set of its edges. If $G=\emptyset$, then we will set $I(G)=(0)$.

\begin{defn}
A (non-empty) graph is called a \textit{star} if all its edges have one vertex in common.
\end{defn}

\begin{defn} Let $G$ be a graph. A subset $C$ of its vertex set is called a \textit{vertex cover} if all edges of $G$ have a vertex in $C$. A vertex cover of $G$ is called \textit{minimal} if it does not properly contain any vertex cover of $G$. A minimal vertex cover of $G$ is called \textit{maximum} if it has maximum cardinality among the minimal vertex covers of $G$.
\end{defn}

\begin{rem}
The unique (hence, the maximum) minimal vertex cover of an empty graph is the empty set.
\end{rem}

\noindent It is well known that the minimal vertex covers of $G$ are the sets of generators of the minimal primes of $I(G)$. Hence $\bight I(G)$ is the cardinality of the maximum minimal vertex covers of $G$.

\begin{defn}
Let $G$ be a graph, and $H$ a subgraph of $G$.
\begin{list}{}{}
\item[$(i)$] If $V(H)=V(G)$, we will say that $G$ is \textit{spanned} by $H$.
\item[$(ii)$] If $C$ is a minimal vertex cover of $G$, we will say that the (possibly empty) set  $C\cap V(H)$ is the vertex cover \textit{induced} by $C$ on $H$.
\end{list}
\end{defn}

\begin{defn} Let $G$ be a graph, and $H_1$ and $H_2$ be two subgraphs of $G$.  If $H_1$ and $H_2$ are vertex-disjoint (i.e. have disjoint vertex sets) and there are a vertex $x_1$ of $H_1$ and a vertex $x_2$ of $H_2$ such that $G = H_1\cup\{x_1x_2\}\cup H_2$, we will say that $H_1$ and $H_2$ are \textit{connected through the edge} $x_1x_2$.
\end{defn}

Given a graph $G$, and a vertex $x$ of $G$, adding a \textit{whisker} to $G$ at $x$ means adding a new vertex $y$ to the vertex set of $G$ and the edge $xy$ to its edge set. The new edge will be referred to as a whisker attached to $x$. A graph obtained from $G$ by adding one or more whiskers to some of its vertices will be called a \textit{whisker graph} on $G$. If there is at least one whisker attached to every vertex of $G$, $G$ will be called \textit{fully whiskered}. These graphs were considered by the second author in \cite{M13b}, were it was proven that the arithmetical rank of their edge ideals is always equal to its big height.
\\[3mm]
\indent In the sequel we will adopt the following notation about a graph $G$. If we add a whisker to $G$, we will denote the new graph by $G'$. If we remove one, two or possibly more edges from $G$, we will denote the resulting graph by $\overline{G}$, $\widetilde{G}$, or $\widehat{G}$, respectively. Moreover, $G^\bullet$ and $G^\vee$ will be alternative notations for $\overline{G}$ and $\widetilde{G}$, respectively. We will often use the corresponding notation for the vertex covers and the subgraphs of these graphs: for example $C'$ will be a vertex cover of $G'$ and $H'$ will be a subgraph of $G'$.

\section{More on maximum minimal vertex covers}

In this section we will present some preliminary technical results regarding the relations between the maximum minimal vertex covers of a graph and those of certain graphs derived from it.

\begin{lem}\label{GGhat}
Let $G$ be a graph and let $x_1$ and $x_2$ be two leaves of $G$ belonging to disjoint edges. Let $\dot{G}$ be the graph obtained from $G$ by identifying $x_1$ and $x_2$. Then
\[
\bight I(\dot{G})\geq \bight I(G)-1.
\]
Moreover
\[
\ara I(\dot{G})\leq \ara I(G).
\]
\end{lem}

\begin{proof}
Let $C$ be a maximum minimal vertex cover of $G$. In order to prove the first inequality, we show that $\dot{G}$ has a minimal vertex cover of cardinality $|C|$ or $|C|-1$. Let $x$ be the vertex of $\dot{G}$ obtained by identifying $x_1$ and $x_2$. For $i=1,2$, let $y_i$ be the only neighbour of $x_i$ in $G$. Then $y_1\neq y_2$.
\\ If $x_1, x_2\notin C$, then $y_1, y_2\in C$. In this case, $C$ is a minimal vertex cover of $\dot{G}$. So suppose that $x_1,x_2\in C$. Then $y_1, y_2\notin C$ and $\dot{C}=(C\setminus\{x_1, x_2\})\cup\{x\}$ is a minimal vertex cover of $\dot{G}$. Finally, suppose that $x_1\in C$, $x_2\notin C$. In this case $y_1\notin C$, $y_2\in C$. If all neighbours of $y_2$ other than $x_2$ belong to $C$, then  $\dot{C}=(C\setminus\{x_1, y_2\})\cup\{x\}$ is a minimal vertex cover of $\dot{G}$. Otherwise so is $\dot{C}=(C\setminus\{x_1\})\cup\{x\}$. This proves the first inequality. For the second inequality, let $S$ be the set of edge monomials of $I(G)$ other than $x_1y_1$ and $x_2y_2$. Then the set of edge monomials of $I(\dot{G})$ is $S\cup\{xy_1, xy_2\}$. Hence, if $q_1,\dots, q_r\in R$ are such that $I(G)=\sqrt{(q_1,\dots, q_r)}$, then $I(\dot{G})=\sqrt{(\overline{q}_1,\dots, \overline{q}_r)}$, where  $\overline{q}$ denotes the polynomial obtained from $q$ by replacing $x_1$ and $x_2$ with $x$.
\end{proof}

\begin{lem}\label{lemma1}
Let $G$ be a graph and $x$ one of its vertices. Let $G'$ be the graph obtained from $G$ by attaching a whisker to $x$. Then
\[
\bight I(G) \leq \bight I(G') \leq \bight I(G) + 1.
\]
Moreover, $\bight I(G) = \bight I(G')$ if and only if $x$ belongs to all maximum minimal vertex covers of $G$ (which are also maximum minimal vertex covers of $G'$).
\end{lem}

\begin{proof}
Let $C$ be a maximum minimal vertex cover of $G$ and $y$ be the other endpoint of the whisker attached to $x$. If $C$ is not a vertex cover of $G'$, then $x\notin C$, so that $C \cup \{y\}$ is a minimal vertex cover of $G'$. Hence
\begin{equation}\label{1}
\bight I(G) \leq \bight I(G').
\end{equation}
Conversely, let $C'$ be a maximum minimal vertex cover of $G'$. If $C'$ is not a minimal vertex cover of $G$, then $C' \setminus \{x\}$ or $C' \setminus \{y\}$ is one. Thus
\[
\bight I(G') \leq \bight I(G) + 1.
\]
Now, equality holds in \eqref{1} if and only if for all maximum minimal vertex covers $C$ of $G$, $C \cup \{y\}$ is not a minimal vertex cover of $G'$, which, in turn, is true if and only if $x \in C$.
\end{proof}

\begin{lem}\label{minimal_induced}
Let $H$ be a subgraph of the graph $G$, and let $x$ be a vertex of $G$ such that $V(H)\cap V(G\setminus H)=\{x\}$. Let $C$ be a minimal vertex cover of $G$ and call $D$ the vertex cover it induces on $H$.
\begin{list}{}{}
\item[$(i)$] If $x\notin C$,  then $D$ is minimal.
\item[$(ii)$] If $x\in C$, then $D$ or $D\setminus\{x\}$ is minimal.
\end{list}
\end{lem}
\begin{proof} $(i)$ If $D$ is empty, then there is nothing to prove. So assume that $D$ is not empty. Let $y\in D$. By assumption, $C\setminus\{y\}$ is not a vertex cover of $G$. Hence there is a neighbour $z$ of $y$ in $G$ such that $z\notin C$. Note that $y\neq x$, since $x\notin D$. On the other hand, $y$ is a vertex of $H$, so that $y\notin V(G\setminus H)$. Hence $yz$ cannot be an edge of $G\setminus H$, i.e., it is an edge of $H$. This edge is left uncovered by $D\setminus\{y\}$, which proves the minimality of $D$.\\
$(ii)$ Suppose that $D$ is not minimal. Then there is $y\in D$ such that $\overline{D}=D\setminus\{y\}$ is a vertex cover of $H$. But $C\setminus\{y\}$ is not a vertex cover of $G$. Hence there is some neighbour $z$ of $y$ in $G$ such that $z\notin C$. The edge $yz$ does not belong to $H$, because it is not covered by $\overline{D}$. Hence it belongs to $G\setminus H$. Thus $y\in V(H)\cap V(G\setminus H)$, i.e., $y=x$. Next we show that $\overline{D}$ is minimal. Let $v\in \overline{D}$. Then $v\neq x$. We prove that $\widetilde{D}=\overline{D}\setminus\{v\}$ is not a vertex cover of $H$. By assumption $C\setminus\{v\}$ is not a vertex cover of $G$. Hence there is a neighbour $w$ of $v$ in $G$ such that $w\notin C$. Now, if $vw$ is an edge of $H$, then $\widetilde{D}$ leaves $vw$ uncovered, which proves that $\overline{D}$ is minimal. Otherwise $vw$ is an edge of $G\setminus H$, but then $v\in V(H)\cap V(G\setminus H)$, which is impossible, since $v\neq x$.
\end{proof}

\begin{lem}\label{lemma1.1}
Let $G$ be a graph and let $H_1$ and $H_2$ be subgraphs of $G$ whose vertex sets have exactly one element $x$ in common and such that $G=H_1\cup H_2$. Suppose that $x$ belongs to all maximum minimal vertex covers of $H_1$ and of $H_2$. Let $D_1$ and $D_2$ be maximum minimal vertex covers of $H_1$ and $H_2$, respectively. Then $D_1\cup D_2$ is a maximum minimal vertex cover of $G$. Moreover, $x$ belongs to all maximum minimal vertex covers of $G$.
\end{lem}

\begin{proof}
Let $C=D_1\cup D_2$. Then $C$ is a vertex cover of $G$.    We first prove that $C$ is  minimal. Let $y\in C$, say $y\in D_1$. Since $D_1$ is minimal, $D_1\setminus\{y\}$ does not cover $H_1$, hence there is a vertex $z$ of $H_1$ such that $yz$ is an edge of $H_1$ and $z\notin D_1$. But then $z\neq x$, so that $z\notin D_2$. Thus $z\notin C$. It follows that $C\setminus\{y\}$ leaves the edge $yz$ uncovered. This proves the minimality of $C$.
\\ We now prove that $C$ is maximum. Set $d_1=\vert D_1\vert$ and $d_2=\vert D_2\vert$. Then $\vert C\vert =d_1+d_2-1$, since $x$ is the only common element of $D_1$ and $D_2$. Let $C^\ast$ be any minimal vertex cover of $G$. We show that $\vert C^\ast\vert\leq\vert C\vert$. Let $D_1^\ast$ and $D_2^\ast$ be the vertex covers induced by $C^\ast$ on $H_1$ and $H_2$, respectively. If $x\not\in C^\ast$, then $D_1^\ast$ and $D_2^\ast$ are disjoint, so that $\vert C^\ast\vert = \vert D_1^\ast\vert +\vert D_2^\ast\vert$. Moreover, $D_1^\ast$ and $D_2^\ast$ are minimal on $H_1$ and $H_2$, respectively. But, by assumption, they are not maximum, whence $\vert D_1^\ast\vert\leq d_1-1$ and $\vert D_2^\ast\vert\leq d_2-1$, so that $\vert C^\ast\vert\leq d_1+d_2-2$. This also shows that
$C^\ast$ is not maximum if $x\notin C^\ast$.\newline
Now suppose that $x\in C^\ast$, and let $i\in\{1,2\}$. If $D_i^\ast$ is minimal, then $\vert D_i^\ast\vert\leq d_i$. Otherwise, by Lemma \ref{minimal_induced} $(ii)$, $D_i^\ast\setminus\{x\}$ is a minimal vertex cover of $H_i$ and, by assumption, it is not maximum. Hence, once again, $\vert D_i^\ast\vert\leq d_i$.  Thus $\vert C^\ast\vert=\vert D_1^\ast\vert+\vert D_2^\ast\vert-1\leq d_1+d_2-1.$
\end{proof}

\begin{lem}\label{lemma1.2}
Let $G$ be a graph formed by two graphs $H_1$ and $H_2$ connected through an edge. Let $x_1$ and $x_2$ be the endpoints of this edge, where, for $i=1,2$, $x_i$ is a vertex of $H_i$. For $i=1,2$, let $D_i$ be a maximum minimal vertex cover of $H_i$ and set $H_i'=H_i\cup\{x_1x_2\}$. Suppose that one of the following conditions holds:
\begin{list}{}{}
\item[$(i)$] $x_1\in D_1$ and, for $i=1,2$, $x_i$ belongs to no maximum minimal vertex cover of $H_i'$;
\item[$(ii)$] $x_1$ belongs to all maximum minimal vertex covers of $H_1$.
\end{list}
Then $D_1\cup D_2$ is a maximum minimal vertex cover of $G$.
\end{lem}

\begin{proof} Set $D=D_1\cup D_2$. Suppose that $(i)$ or $(ii)$ holds. Then $x_1\in D$, so that $D$ is a vertex cover of $G$. It is also minimal, because $H_1$ and $H_2$ are vertex-disjoint, and  $D_i$ is minimal on $H_i$, for $i=1,2$. We prove that it is also maximum. Let $C$ be a minimal vertex cover of $G$ and, for $i=1,2$, let $E_i$ be the vertex cover it induces on $H_i$. Then $E_1$ and $E_2$ are disjoint and $C=E_1\cup E_2$. \newline
Suppose that $(i)$ holds. Let $i\in\{1,2\}$. If $E_i$ is minimal as a vertex cover of $H_i$, then $\vert E_i\vert\leq \vert D_i\vert$. Otherwise, by Lemma \ref{minimal_induced} $(ii)$, $x_i\in E_i$ and $E_i\setminus\{x_i\}$ is minimal, whence $E_i$ is a minimal vertex cover of $H_i'$. By assumption it is not maximum, i.e., $\bight I(H_i')\geq\vert E_i\vert+1$.  Moreover, the same assumption, together with Lemma \ref{lemma1}, implies that
$$\bight I(H_i')=\bight I(H_i)+1=\vert D_i\vert+1.$$
Hence $\vert E_i\vert\leq\vert D_i\vert.$ Thus in any case we have
$$\vert C\vert=\vert E_1\vert+\vert E_2\vert\leq \vert D_1\vert+\vert D_2\vert=\vert D\vert.$$
Now suppose that $(ii)$ holds. Then, in view of Lemma \ref{lemma1}, $D_1$ is a maximum minimal vertex cover of $H_1'$. If $x_1\in C$, then $E_1$ is a minimal vertex cover of $H_1'$, whence $\vert E_1\vert\leq \vert D_1\vert$. On the other hand, $E_2$ is a minimal vertex cover of $H_2$: otherwise, by Lemma \ref{minimal_induced} $(ii)$, so would be $E_2\setminus\{x_2\}$, and thus $C\setminus\{x_2\}=E_1\cup E_2\setminus\{x_2\}$ would be a minimal vertex cover of $G$, against the minimality of $C$. The minimality of $E_2$ implies that $\vert E_2\vert\leq\vert D_2\vert$. If $x_1\not\in C$, then $x_2\in C$ and $E_1$ is a non-maximum minimal vertex cover of $H_1$, whereas $E_2$ is a minimal vertex cover of $H_2'$. Hence $\vert E_1\vert\leq \vert D_1\vert-1$ and, in view of Lemma \ref{lemma1}, $\vert E_2\vert\leq \vert D_2\vert+1$. Thus we always have $\vert C\vert\leq\vert D_1\vert+\vert D_2\vert=\vert D\vert.$
\end{proof}

\section{Graphs with pairwise disjoint cycles}

In this section we present our main result. Its proof, which will be performed by induction on the number of edges, essentially rests on the way in which the maximum minimal vertex covers of a graph $G$ \textit{with pairwise disjoint cycles} (i.e. whose cycles are pairwise vertex-disjoint)  behave upon removal of special edges. Lemma \ref{lemma3} will give a complete classification of the possible cases.

\begin{defn}
Let $G$ be a graph and $x$ one of its vertices. A neighbour of $x$ in $G$ will be called \textit{free} if it does not lie on a cycle through $x$.
\end{defn}

\begin{defn} Let $G$ be a graph and $x$ be one of its vertices. Let $C$ be a minimal vertex cover of $G$ such that $x \notin C$. Then a neighbour $y$ of $x$ will be called \textit{redundant} (\textit{with respect to $C$}) if
\[
\{y\} \cup N(y) \setminus \{x\} \subset C.
\]
\end{defn}

\begin{rem}\label{remark2}
Let $C$ be a minimal vertex cover of $G$ such that $x \notin C$. Then $C$ contains a redundant neighbour $y$ of $x$ if and only if $C \setminus \{y\}$ is a minimal vertex cover of $G \setminus \{xy\}$. Moreover, in this case $C \cup \{x\}$ is not a minimal vertex cover of $G$.
\end{rem}

\begin{lem}\label{lemma3}
Let $G$ be a graph with pairwise disjoint cycles, and let $x$ be one of its vertices. Suppose that there is a maximum minimal vertex cover $C$ of $G$ such that $x \notin C$, and that for all minimal vertex covers with this property, there is a redundant neighbour of $x$. Then in $C$ there is either a free redundant neighbour $y$ of $x$ or a non-free redundant neighbour $z_1$ of $x$ for which the statement $1)$ or $2)$, respectively, is true.\newline
$1)$ Set $\overline{G} = G \setminus \{xy\}$. Then one of the following conditions holds:
\begin{itemize}
\item[$(a)$] $\bight I(\overline{G}) = \bight I(G) - 1$;
\item[$(b)$] if $\overline{H}$ is the connected component of $y$ in $\overline{G}$, and $\overline{K} = \overline{G} \setminus \overline{H}$, then $x$ belongs to all maximum minimal vertex covers of $\overline{K}$ (so that, in particular,  $\overline{K}$ is not empty).
\end{itemize}
$2)$ Set $\overline{G} = G \setminus \{xz_1\}$. Then call $z_2$ the other non-free neighbour of $x$ and set $\widetilde{G} = G \setminus \{xz_1,xz_2\}$. Then one of the following conditions holds:
\begin{itemize}
\item[$(c)$] $\bight I(\overline{G}) = \bight I(G) - 1$;
\item[$(d)$] $\bight I(\widetilde{G}) = \bight I(G) - 1$;
\item[$(e)$] if $\widetilde{H}$ is the connected component of $z_1$ (and $z_2$) in $\widetilde{G}$, and $\widetilde{K} = \widetilde{G} \setminus \widetilde{H}$, then $x$ belongs to all maximum minimal vertex covers of $\widetilde{K}$ (so that, in particular,  $\widetilde{K}$ is not empty).
\end{itemize}
Moreover, if for all choices of $C$ there is a free redundant neighbour of $x$, then (a) or (b) is true for some free redundant neighbour $y$ of $x$.
\end{lem}

\begin{proof}
Set $c = \bight I(G)$. Let $C$ be a maximum minimal vertex cover of $G$ such that $x \notin C$. Let $S$ be the set of redundant free neighbours of $x$ with respect to $C$. We proceed by induction on $|S|$.
\\ First suppose that $|S|=0$. Then, according to the assumption, some non-free neighbour $z_1$ of $x$ is redundant with respect to $C$. Then, by Remark \ref{remark2}, $C \setminus \{z_1\}$ is a minimal vertex cover of $\overline{G} = G \setminus \{xz_1\}$, so that $\bight I(\overline{G}) \geq c-1.$ Suppose that $(c)$ is not true. Then
\begin{equation}\label{GG'}
\bight I(\overline{G}) \geq \bight I(G).
\end{equation}
Let $\overline{C}$ be a maximum minimal vertex cover of $\overline{G}$. Then $\overline{C}$ or $\overline{C} \setminus \{x\}$ or $\overline{C} \setminus \{z_2\}$ is a minimal vertex cover of $\widetilde{G}$, whence $\bight I(\widetilde{G}) \geq \bight I(\overline{G}) - 1 \geq c-1$. Suppose that also $(d)$ is false. Then
\begin{equation}\label{G'G''}
\bight I(\widetilde{G}) \geq \bight I(G).
\end{equation}
\\ Let $D$ and $E$ be the covers induced by $C$ on $\widetilde{K}$ and $\widetilde{H}$, respectively (see Figure \ref{F2.lemma4.4}). Note that all neighbours of $x$ lying in $\widetilde{K}$ are free, whence, by assumption, in $D$ there are no redundant neighbours of $x$.  Let $\widetilde{C}$ be a maximum minimal vertex cover of $\widetilde{G}$. Let $\widetilde{D}$ and $\widetilde{E}$ be the covers induced by $\widetilde{C}$ on $\widetilde{K}$ and $\widetilde{H}$, respectively.  Since $\widetilde{K}$ and $\widetilde{H}$ are vertex-disjoint, $D, E$ and $\widetilde{D}, \widetilde{E}$ are disjoint. Moreover, $\widetilde{G}$ is the vertex-disjoint union of $\widetilde{H}$ and $\widetilde{K}$, so that $\widetilde{C}$ is the disjoint union of $\widetilde{D}$ and $\widetilde{E}$, and $\widetilde{D}$, $\widetilde{E}$ are maximum minimal vertex covers of $\widetilde{K}$ and $\widetilde{H}$, respectively. Thus $|\widetilde{C}|=|\widetilde{D}|+|\widetilde{E}|$. Since $G$ is spanned by $\widetilde{H}\cup\widetilde{K}$, we also have that $C=D\cup E$, and $|C|=|D|+|E|$.  But, in view of \eqref{G'G''}, $|\widetilde{C}| \geq |C|$. Since $x\notin C$ and $x$ is the only vertex that $\widetilde{K}$ has in common with $G\setminus\widetilde{K}=\widetilde{H}\cup\{xz_1,xz_2\}$, by Lemma \ref{minimal_induced} $(i)$, $D$ is minimal.  Hence $|D| \leq |\widetilde{D}|$. First suppose that $x \notin \widetilde{C}$. In this case the inequality $|D| < |\widetilde{D}|$ would imply that $\widetilde{D} \cup E$ is a minimal vertex cover of $G$ of cardinality greater than $c$, which is impossible. Hence $|D| = |\widetilde{D}|$, so that $|E| \leq |\widetilde{E}|$. Note that, for $i=1,2$, if $z_i \in \widetilde{E}$, then not all  neighbours of $z_i$ lying in $\widetilde{H}$ (i.e., other than $x$) belong to $\widetilde{E}$. Hence one of the following cases occurs:
\begin{itemize}
\item[$\bullet$] $\{z_1, z_2\} \subset \widetilde{E}$, in which case $D \cup \widetilde{E}$ is a maximum minimal vertex cover of $G$ without $x$ and without redundant neighbours of $x$, against our assumption;
\item[$\bullet$] $\{z_1, z_2\} \not\subset \widetilde{E}$, in which case $D \cup \{x\} \cup \widetilde{E}$ is a minimal vertex cover of $G$. But this, again, is impossible, since its cardinality is greater than $c$.
\end{itemize}
We thus conclude that $x \in \widetilde{C}$ for all maximum minimal vertex covers $\widetilde{C}$ of $\widetilde{G}$. Now, if $\widetilde{D}^{\ast}$ and $\widetilde{E}^{\ast}$ are maximum minimal vertex covers of $\widetilde{K}$ and $\widetilde{H}$, respectively, then their union is a maximum minimal vertex cover of $\widetilde{G}$. This shows that $(e)$ is true, and completes the proof of the induction basis.
\\ Now suppose that $|S| \geq 1$.   Suppose that $(a)$ is false for all choices of $y\in S$, where $\overline{G} = G \setminus \{xy\}$. Then \eqref{GG'} holds for all choices of $y\in S$. Let $y \in S$ be fixed.  Let $\overline{D}$ be a maximum minimal vertex cover of $\overline{K}$, and $\overline{E}$ a maximum minimal vertex cover of $\overline{H}$ (see Figure \ref{F1.lemma4.4}). Since $\overline{G}$ is the vertex-disjoint union of $\overline{H}$ and $\overline{K}$, $\overline{C} = \overline{D} \cup \overline{E}$ is a maximum minimal vertex cover of $\overline{G}$ and $|\overline{C}| = |\overline{D}| + |\overline{E}|$.

\begin{figure}[ht!]\centering
\begin{pspicture}(-2,-0.7)(4,1.7)
\psline[linewidth=1.5pt,linestyle=dotted](0,0)(1,0)
\psline[linestyle=dashed](1,0)(3,0)
\psline[linestyle=dashed](3,0)(3.71,0.71)
\psline[linestyle=dashed](3,0)(3.71,-0.71)
\psline(0,0)(-0.71,0.71)
\psline(0,0)(-0.71,-0.71)
\psline(-0.97,1.67)(-1.67,0.97)
\psline(-0.97,1.67)(-0.71,0.71)
\psline(-0.71,0.71)(-1.67,0.97)

\psdots(0,0)
\psdots[dotsize=3pt 3,dotstyle=Bo](1,0)
\psdots[dotsize=3pt 3,dotstyle=Bo](2,0)
\psdots(3,0)
\psdots[dotsize=3pt 3,dotstyle=Bo](3.71,-0.71)
\psdots[dotsize=3pt 3,dotstyle=Bo](3.71,0.71)
\psdots[dotsize=3pt 3,dotstyle=Bo](-0.71,0.71)
\psdots[dotsize=3pt 3,dotstyle=Bo](-0.71,-0.71)
\psdots[dotsize=3pt 3,dotstyle=Bo](-1.67,0.97)
\psdots(-0.97,1.67)

\uput[300](0,0){$x$}
\uput[300](1,0){$y$}
\rput(1,1){$\overline{G}$}
\rput(2.5,-0.7){$\overline{H}$}
\rput(-1,0){$\overline{K}$}
\end{pspicture}
\caption{}\label{F1.lemma4.4}
\end{figure}
\noindent {\small\textit{In Figure \ref{F1.lemma4.4} the edges of $\overline{H}$ are dashed and the empty dots represent the vertices of the vertex cover $C$.}}
\\\\
\noindent If $x \in \overline{D}$ for all choices of $\overline{D}$, then $(b)$ is true.
\\ So suppose that $x \notin \overline{D}$ for some $\overline{D}$. If $D$ is the cover induced by $C$ on $\overline{K}$, we then have $|\overline{D}| = |D|$, because $D$ and $\overline{D}$ are interchangeable in $C$ and $\overline{C}$. Let $E$ be the cover induced by $C$ on $\overline{H}$. Then $|C| = |D| + |E|$. Since, in view of \eqref{GG'}, $|\overline{C}| \geq |C|$, it follows that $|\overline{E}| \geq |E|$. If $y \in \overline{E}$ for some choice of $\overline{E}$, then not all neighbours of $y$ lying in $\overline{H}$ (i.e., other than $x$) belong to $\overline{E}$. Moreover, $y$ is the only neighbour of $x$ lying in $\overline{H}$: this follows from the fact that $y$ is free.  Hence $D \cup \overline{E}$ is a maximum minimal vertex cover of $G$ in which the set of redundant free neighbours of $x$ is $S \setminus \{y\}$, so that induction applies.
\\ Now suppose that, for all $y$ in $S$, we have that $x \notin \overline{D}$ for some choice of $\overline{D}$ and $y \notin \overline{E}$ for all choices of $\overline{E}$. Call $y_1,\dots, y_k$ the elements of $S$ and, for all $i=1,\dots,k$, let $\overline{H}_i$ be the connected component of $y_i$ in $\overline{G}_i = G \setminus \{xy_i\}$, and call $\overline{D}_i$ and $\overline{E}_i$ some maximum minimal vertex covers of $\overline{K}_i = \overline{G}_i \setminus \overline{H}_i$ and $\overline{H}_i$, respectively, where $x \notin \overline{D}_i$, and $y_i \notin \overline{E}_i$. Moreover, let $E_i$ be the cover induced by $C$ on $\overline{H}_i$. Note that the subgraphs $\overline{H}_i$ are pairwise vertex-disjoint: if $z$ were a common vertex of $\overline{H}_i$ and $\overline{H}_j$, with $i\neq j$, then $y_i$ and $y_j$ would lie on a cycle through $x$ and $z$, and would therefore not be free neighbours of $x$. Since $x \notin C$, $E_i$ is also the cover induced by $C$ on $\overline{H}_i \cup \{xy_i\}$. Recall that, for all $i=1,\dots,k$, $|\overline{E}_i| \geq |E_i|$. Set
\[
L = G \setminus \bigcup_{i=1}^k(\overline{H}_i \cup\{xy_i\}).
\]
Let $F$ be the vertex cover induced by $C$ on $L$. Then $C = F \cup (\bigcup_{i=1}^k E_i)$. If in $L$ there are no redundant neighbours of $x$ with respect to $F$, then
\[
F \cup \{x\} \cup \left( \bigcup_{i=1}^k \overline{E}_i \right)
\]
is a minimal vertex cover of $G$ of cardinality greater than $c$, which is impossible. We thus conclude that in $L$ there is some redundant neighbour of $x$ with respect to $F$ (and with respect to $C$), which is necessarily not free. We may call it $z_1$. Let $\widetilde{E}$ be the vertex cover induced by $C$ on $\widetilde{H}$ and let $\widetilde{D}$ be the vertex cover induced by $C$ on $\widetilde{K}$ (see Figure \ref{F2.lemma4.4}). Then $C=\widetilde{D}\cup\widetilde{E}$  and $\widetilde{D}$, $\widetilde{E}$ are disjoint, so that $\widetilde{D}=C\setminus\widetilde{E}$.

\begin{figure}[ht!]\centering
\begin{pspicture}(-2,-1.3)(4,2.4)
\psline(0,0)(-0.71,0.71)
\psline(0,0)(-0.71,-0.71)
\psline(-0.97,1.67)(-1.67,0.97)
\psline(-0.97,1.67)(-0.71,0.71)
\psline(-0.71,0.71)(-1.67,0.97)
\psline[linewidth=1.5pt,linestyle=dotted](0,0)(0.5,-0.87)
\psline[linestyle=dashed](0.5,-0.87)(1.36,-1.37)
\psline[linestyle=dashed](1.36,-1.37)(1.86,-0.5)
\psline[linewidth=1.5pt,linestyle=dotted](0,0)(1,0)
\psline[linestyle=dashed](1,0)(1.86,-0.5)
\psline[linestyle=dashed](1,0)(1.71,0.71)
\psline[linestyle=dashed](1.71,0.71)(2.42,1.42)
\psline[linestyle=dashed](2.42,1.42)(2.42,2.42)
\psline[linestyle=dashed](2.42,1.42)(3.42,1.42)

\psdots(0,0)
\psdots[dotsize=3pt 3,dotstyle=Bo](1,0)
\psdots[dotsize=3pt 3,dotstyle=Bo](-0.71,0.71)
\psdots[dotsize=3pt 3,dotstyle=Bo](-0.71,-0.71)
\psdots[dotsize=3pt 3,dotstyle=Bo](-1.67,0.97)
\psdots(-0.97,1.67)
\psdots[dotsize=3pt 3,dotstyle=Bo](0.5,-0.87)
\psdots(1.36,-1.37)
\psdots[dotsize=3pt 3,dotstyle=Bo](1.86,-0.5)
\psdots[dotsize=3pt 3,dotstyle=Bo](1.71,0.71)
\psdots(2.42,1.42)
\psdots[dotsize=3pt 3,dotstyle=Bo](2.42,2.42)
\psdots[dotsize=3pt 3,dotstyle=Bo](3.42,1.42)

\uput[65](0,0){$x$}
\uput[120](1,0){$z_1$}
\uput[225](0.5,-0.87){$z_2$}
\rput(0.5,1.5){$\widetilde{G}$}
\rput(2.8,0.5){$\widetilde{H}$}
\rput(-1,0){$\widetilde{K}$}
\end{pspicture}
\caption{}\label{F2.lemma4.4}
\end{figure}
\noindent {\small\textit{In Figure \ref{F2.lemma4.4} the edges of $\widetilde{H}$ are dashed and the dotted edges do not belong to $\widetilde{G}$.}}
\\\\
Furthermore, for all $i=1,\dots, k$, $\widetilde{H}$ is vertex-disjoint from $\overline{H}_i$ (and thus $\widetilde{E}$ is disjoint from $E_i$), because otherwise  $y_i$ would not be a free neighbour of $x$. Therefore
\[
\widetilde{D} = C\setminus\widetilde{E}= (F \setminus \widetilde{E}) \cup \left( \bigcup_{i=1}^k E_i \right).
\]
Since $x \notin \widetilde{D}$, and $x$ is the only common vertex of $\widetilde{K}$ and $G\setminus\widetilde{K}=\widetilde{H}\cup\{xz_1, xz_2\}$, from Lemma \ref{minimal_induced} $(i)$ it follows that $\widetilde{D}$ is a minimal vertex cover of $\widetilde{K}$. Let $\widetilde{D}^{\ast}$ be a maximum minimal vertex cover of $\widetilde{K}$, so that $|\widetilde{D}| \leq |\widetilde{D}^{\ast}|$. Suppose that $x \notin \widetilde{D}^{\ast}$. In this case replacing $\widetilde{D}$ by $\widetilde{D}^{\ast}$ in $C=\widetilde{D}\cup\widetilde{E}$ produces a minimal vertex cover of $G$, and the maximality of  $C$ implies $|\widetilde{D}^{\ast}| \leq |\widetilde{D}|$, so that equality holds, and $\widetilde{D}$ is maximum. Now, since $x\notin C$, we have that $z_1,z_2\in \widetilde{E}$. Thus in $F \setminus \widetilde{E}$ there are no redundant neighbours of $x$. This implies that
\[
(F \setminus \widetilde{E}) \cup \{x\} \cup \left( \bigcup_{i=1}^k \overline{E}_i \right)
\]
is a minimal vertex cover of $\widetilde{K}$ of cardinality greater than $|\widetilde{D}|$, which contradicts the maximality of $\widetilde{D}$. This shows that $x \in \widetilde{D}^{\ast}$ for all maximum minimal vertex covers of $\widetilde{K}$, i.e., $(e)$ holds.
\end{proof}

\begin{cor}\label{remark}
\begin{itemize}
\item[$(i)$] If condition $(b)$ of Lemma $\ref{lemma3}$ holds for the graph $G$ with respect to $y$, then every maximum minimal vertex cover of $\overline{G} = G \setminus \{xy\}$ contains $x$ and is a maximum minimal vertex cover of $G$.
\item[$(ii)$] If condition $(e)$ of Lemma $\ref{lemma3}$ holds for the graph $G$ with respect to $z_1$, then  every maximum minimal vertex cover of $\widetilde{G}$ contains $x$ and is a maximum minimal vertex cover of $\overline{G} = G \setminus \{xz_1\}$.
\end{itemize}
\end{cor}
\begin{proof}
 Suppose that condition $(b)$ holds for $G$. Since $\overline{G}$ is the vertex-disjoint union of $\overline{H}$ and $\overline{K}$, the maximum minimal vertex covers of $\overline{G}$ are the unions of a maximum minimal vertex cover of $\overline{H}$ and a maximum minimal vertex cover of $\overline{K}$. This implies that all maximum minimal vertex covers of $\overline{G}$ contain $x$. Moreover, since in $G$ the subgraphs $\overline{H}$ and $\overline{K}$ are connected through the edge $xy$, the second part of claim $(i)$ follows from Lemma \ref{lemma1.2} $(ii)$.\\
The first part of claim $(ii)$ follows as above from the fact that $\widetilde{G}$ is the vertex-disjoint union of $\widetilde{H}$ and $\widetilde{K}$. For the second part, note that $\widetilde{G}=\overline{G}\setminus\{xz_2\}$, and $\widetilde{H}$ and $\widetilde{K}$ are connected in $\overline{G}$ through the edge $xz_2$. Hence the claim once again follows from Lemma \ref{lemma1.2} $(ii)$.
\end{proof}

\begin{lem}\label{fundamental}
Let $G$ be a non-empty graph with pairwise disjoint cycles. Suppose that all edges of $G$ that do not belong to a cycle are terminal. Then every connected component of $G$ is a star, a cycle, or a whisker graph on a cycle.
\end{lem}

\begin{proof}
Every connected component of $G$ fulfils the same assumption as $G$. Hence it suffices to prove the claim in the case where $G$ is connected. We first prove that $G$ has at most one cycle. Let $T_1$ and $T_2$ be cycles of $G$. Let $a$ be a vertex of $T_1$ and $b$ a vertex of $T_2$, where $a \neq b$. Since $G$ is connected, in $G$ there is a path with endpoints $a$ and $b$, say $L: a = c_0 c_1 \dots c_{n-1} c_n = b$. Then, for all $i=0,\dots,n$, $c_i$ is not a terminal vertex, so that none of the edges of $L$ is a terminal edge. Hence each of them must lie on a cycle. Since the cycles of $G$ are pairwise disjoint, $a c_1$ must lie on $T_1$. Let $k$ be the maximum index such that the edge $c_k c_{k+1}$ of $L$ lies on $T_1$. If $k=n-1$, then $b$ lies on $T_1$, whence $T_1=T_2$. So assume that $k \leq n-2$. Then $c_{k+1} c_{k+2}$ is contained in a cycle distinct from $T_1$. But this is impossible, because $c_{k+1}$ lies on $T_1$ and the cycles of $G$ are pairwise disjoint. This shows that $G$ has at most one cycle. Let $H$ be this cycle, and suppose that $H\neq G$. Since none of the edges of $G\setminus H$ belongs to a cycle, these are all terminal edges. Since $G$ is connected, they all have an endpoint on $H$. But then $G$ is a whisker graph on a cycle.
\end{proof}

\begin{lem}\label{decomposition}
Let $G$ be a graph with pairwise disjoint cycles and in which there is at least one edge that is not terminal and does not belong to a cycle. Then there are two non-empty subgraphs $G_1$ and $G_2$ of $G$ that are connected through this edge.
\end{lem}

\begin{center}
\begin{figure}[ht!]
\begin{pspicture}(-1.4,-0.6)(4.8,2.5)
\psline(0,0)(-1,0)
\psline(-1,0)(-0.32,0.73)
\psline[linewidth=1.5pt,linestyle=dotted](1.68,0.73)(0.68,0.73)
\psline(0.68,0.73)(0,0)
\psline(0.68,0.73)(0.19,1.61)
\psline(0.17,2.61)(0.19,1.61)
\psline(0.19,1.61)(-0.72,2.01)
\psline(0.68,0.73)(-0.32,0.73)
\psdots(-1,0)
\psdots(0,0)
\psdots(1,0)
\psdots(1.68,0.73)
\psdots(0.68,0.73)
\psdots(-0.32,0.73)
\psdots(0.19,1.61)
\rput(-1.2,1.3){$G_1$}

\psline[linewidth=1pt,linestyle=dashed](1,0)(1.68,0.73)
\psline[linewidth=1pt,linestyle=dashed](2.55,0.24)(3.43,-0.22)
\psline[linewidth=1pt,linestyle=dashed](3.43,-0.22)(3.45,0.68)
\psline[linewidth=1pt,linestyle=dashed](3.45,0.68)(2.55,0.24)
\psline[linewidth=1pt,linestyle=dashed](2.55,0.24)(1.81,-0.44)
\psline[linewidth=1pt,linestyle=dashed](1.68,0.73)(2.55,0.24)
\psline[linewidth=1pt,linestyle=dashed](1.68,0.73)(1.49,1.71)
\psline[linewidth=1pt,linestyle=dashed](1.68,0.73)(2.51,1.29)
\psline[linewidth=1pt,linestyle=dashed](2.51,1.29)(2.68,2.28)
\psline[linewidth=1pt,linestyle=dashed](2.51,1.29)(3.3,1.91)
\psline[linewidth=1pt,linestyle=dashed](2.51,1.29)(3.51,1.26)
\psline[linewidth=1pt,linestyle=dashed](1.49,1.71)(0.82,2.46)
\psline[linewidth=1pt,linestyle=dashed](0.82,2.46)(1.91,2.62)
\psline[linewidth=1pt,linestyle=dashed](1.91,2.62)(1.49,1.71)
\psline[linewidth=1pt,linestyle=dashed](3.43,-0.22)(4.33,0.22)
\psline[linewidth=1pt,linestyle=dashed](3.43,-0.22)(4.35,-0.63)
\psdots(-0.72,2.01)
\psdots(0.17,2.61)
\psdots(1.49,1.71)
\psdots(2.51,1.29)
\psdots(2.55,0.24)
\psdots(1.81,-0.44)
\psdots(3.43,-0.22)
\psdots(3.45,0.68)
\psdots(2.68,2.28)
\psdots(3.3,1.91)
\psdots(3.51,1.26)
\psdots(0.82,2.46)
\psdots(1.91,2.62)
\psdots(4.33,0.22)
\psdots(4.35,-0.63)
\psdots(1.68,0.73)
\rput(4,2){$G_2$}

\end{pspicture}
\caption{} \label{G1G2}
\end{figure}
\end{center}
\noindent {\small\textit{In Figure \ref{G1G2} the edges of $G_2$ are dashed and the dotted line is the edge connecting $G_1$ and $G_2$.}}
\begin{proof}
Let $ax$ be an edge of $G$ that is not terminal and does not belong to a cycle. Let $\overline{G} = G \setminus \{ax\}$, and let $G_1$ be the connected component of $a$ in $\overline{G}$. Since $a$ is not a terminal vertex, i.e., it has a neighbour other than $x$, $G_1$ is not empty. Moreover, set $G_2 = \overline{G} \setminus G_1$. Then $G_1$ and $G_2$ are vertex-disjoint, and $G=G_1\cup\{ax\}\cup G_2$. Since $ax$ does not belong to a cycle of $G$, the vertices $a$ and $x$ are not connected in $\overline{G}$, whence $x$ is not a vertex of $G_1$, i.e., it is a vertex of $G_2$. Since $x$ is not a terminal vertex, it has a neighbour $b \neq a$. But $b$ cannot be connected to $a$ in $\overline{G}$, so that $b$ is a vertex of $G_2$, i.e., $bx$ is an edge of $G_2$. This proves that $G_2$ is not empty, and completes the proof.
\end{proof}

\begin{thm}\label{main}
Let $G$ be a graph with pairwise disjoint cycles. Let $n$ be the number of its cycles. Then
\[
\ara I(G) \leq \bight I(G) + n.
\]
\end{thm}

\begin{proof}
Suppose that a graph $G$ with pairwise disjoint cycles has a cycle in which at least one vertex $x$ has degree 2. Call $y_1$ and $y_2$ the neighbours of $x$, which lie on the same cycle. Let $L$ be the graph obtained by replacing the edges $xy_1$ and $xy_2$ with $x_1y_1$ and $x_2y_2$, where $x_1$ and $x_2$ are new distinct vertices, both terminal. Then $G$ is obtained from $L$ by gluing together the leaves $x_1$ and $x_2$, i.e., with respect to the notation used in Lemma \ref{GGhat}, $G=\dot{L}$. Moreover, the cycles of $L$ are still pairwise disjoint, and, if $n$ is the number of cycles in $G$, the number of cycles in $L$ is $n-1$. Suppose that the claim of the theorem is true for $L$. Then it is also true for $G$, since, in view of Lemma \ref{GGhat},
$$\ara I(G)\leq \ara I(L)\leq\bight I(L)+n-1\leq\bight I(G)+n.$$
Hence, by descending induction on the number of cycles containing some vertex of degree 2,  it suffices to prove the claim in the case where $G$ has no such cycles.
\\ Now suppose that, in addition to this condition, every edge of $G$ that is not terminal and does not belong to a cycle has at least one whisker at each endpoint. In this case it is straightforward to verify that $G$ is a fully whiskered graph. Then the claim is known to be true, since, according to \cite{M13b}, Corollary 4.2, in this case we have $\ara I(G)=\bight I(G)$.
\\ Hence we may assume that
\begin{itemize}
\item[$(i)$] no vertex lying on a cycle of $G$ has degree 2;
\item[$(ii)$] there is an edge of $G$ that is not terminal, does not belong to a cycle, and that has no whisker attached to one of its endpoints.
\end{itemize}
 We want to prove the theorem by induction on the number of edges of $G$. Let $ax$ be any edge of $G$ that is not terminal and does not belong to a cycle, and consider the decomposition  $G=G_1\cup\{ax\}\cup G_2$ described in the proof of Lemma \ref{decomposition} (and depicted in Figure \ref{G1G2}). We may assume that one of the endpoints of $ax$ has no whisker attached. For all $i\in\{1,2\}$, let $G'_i= G_i\cup\{ax\}$.  Obviously, in $G_1$, in $G_2$, in $G'_1$ and $G'_2$ the cycles are pairwise disjoint. Hence induction applies to all these graphs. If a graph does not fulfil condition $(i)$, it can always be reduced to a graph fulfilling $(i)$ by means of the above construction, which leaves the number of edges unchanged. On the other hand, in view of Lemma \ref{fundamental}, a graph fulfilling the above condition $(i)$ but not $(ii)$ is fully whiskered (this includes the case of a star, which is a whisker graph on an isolated vertex). This provides the induction basis. \\
In many cases considered in this proof the induction step will be performed as follows. We show that $G$ can be decomposed as the union of two non-empty graphs $A$ and $B$ having exactly one vertex in common and such that for some maximum minimal vertex cover $D$ of $A$ and $E$ of $B$, $D$ and $E$ are disjoint and $D \cup E$ is a minimal vertex cover of $G$. Then $I(G) = I(A) + I(B)$, so that $\ara I(G) \leq \ara I(A) + \ara I(B)$, and, moreover, $\bight I(G) \geq \bight I(A) + \bight I(B)$. On the other hand, if $r$ and $s$ are the numbers of cycles contained in $A$ and $B$, respectively, then $n=r+s$ and, by induction, $\bight I(A) \geq \ara I(A) - r$ and $\bight I(B) \geq \ara I(B) - s$, so that
\[
\bight I(G) \geq \ara I(A) - r + \ara I(B) - s \geq \ara I(G) - n,
\]
whence the claim follows. In the first part of the proof, the graphs $A$ and $B$ will coincide with $G_1$, $G_2$ or with $G'_1$, $G_2$, or with $G_1$, $G'_2$.
\\ Let $C_1$ and $C_2$  be maximum minimal vertex covers of $G_1$ and $G_2$, respectively.   Moreover, let $n_1$ be the number of cycles contained in $G_1$ and $n_2$ be the number of cycles contained in $G_2$. Then $n=n_1+n_2$ is the number of cycles contained in $G$.
\\
If $a\in C_1$ for all choices of $C_1$, then by Lemma \ref{lemma1.2} $(ii)$, $C_1\cup C_2$ is a maximum minimal vertex cover of $G$.
\\
Similarly, if $x\in C_2$ for all choices of $C_2$, then  $C_1\cup C_2$ is a maximum minimal vertex cover of $G$.\\
Now suppose that, for some choice of $C_1$ and $C_2$, we have  $a \notin C_1$ and $x \notin C_2$.  In this case, in view of Lemma \ref{lemma1}, $C'_1=C_1\cup\{ x\}$ is a maximum minimal vertex cover of $G'_1$ and $C'_2=C_2\cup\{a\}$ is a maximum minimal vertex cover of $G'_2$. In particular we have
\begin{eqnarray}\label{seven}
\bight I(G'_1)=\bight I(G_1)+1,\ \mbox{ and }\ \bight I(G'_2)=\bight I(G_2)+1.
\end{eqnarray}
In the sequel, let $C_1$ denote any maximum minimal vertex cover of $G_1$ such that $a\notin C_1$ and $C_2$ any maximum minimal vertex cover of $G_2$ such that $x\notin C_2$.
If, for some choice of $C_1$, $a$ has no redundant neighbours in $C_1$, then $C_1\cup C'_2$ is a minimal vertex cover of $G$ (see Figure \ref{case1}). Similarly, if, for some choice of $C_2$, $x$ has no redundant neighbours in $C_2$, then $C'_1\cup C_2$ is a minimal vertex cover of $G$.
\\ So assume that for all choices of  $C_1$ and $C_2$, $a$ and $x$ have some redundant neighbour. Then $G_1$ and $G_2$ both fulfil the assumption of Lemma \ref{lemma3}, with respect to the neighbours of $a$ and $x$, respectively. Hence one of conditions $(a)-(e)$ is true for $G_1$ and the same applies to $G_2$.
\begin{center}
\begin{figure}[ht!]
\begin{pspicture}(-1.9,-1)(2,1)
\psline(0,0)(1,0)
\psline[linewidth=1pt,linestyle=dashed](1,0)(1.71,0.71)
\psline[linewidth=1pt,linestyle=dashed](1,0)(1.71,-0.71)
\psline(-0.71,0.71)(-1.41,0)
\psline(-1.41,0)(-0.71,-0.71)
\psline(0,0)(-0.71,0.71)
\psline(-0.71,-0.71)(0,0)
\psdots[dotsize=3pt 3,dotstyle=Bo](0,0)
\psdots(1,0)
\psdots[dotsize=3pt 3,dotstyle=Bo](1.71,0.71)
\psdots[dotsize=3pt 3,dotstyle=Bo](1.71,-0.71)
\psdots[dotsize=3pt 3,dotstyle=Bo](-0.71,-0.71)
\psdots[dotsize=3pt 3,dotstyle=Bo](-0.71,0.71)
\psdots(-1.41,0)
\uput[290](0,0){$a$}
\uput[250](1,0){$x$}
\rput(-1.5,0.5){$G_1$}
\rput(2.3,0.5){$G_2$}
\end{pspicture}
\caption{} \label{case1}
\end{figure}
\end{center}
\noindent {\small\textit{In Figure \ref{case1} the edges of $G_2$ are dashed.}}
\\\\
Let $C'_1$ and $C'_2$ be arbitrary maximum minimal vertex covers of $G'_1$ and $G'_2$, respectively.
\\ Now, if $a\in C'_1$ for some choice of $C'_1$, then $C'_1$ would be a vertex cover of $G_1$, but, in view of \eqref{seven}, not a minimal one. Then, by Lemma \ref{minimal_induced} $(ii)$, $C'_1\setminus\{a\}$ would be a minimal vertex cover of $G_1$, maximum by \eqref{seven}. But this cover does not contain any redundant neighbours of $a$, because otherwise $C'_1$ would not be minimal. This contradicts our present assumption. Thus $a$ does not belong to $C'_1$, for all choices of $C'_1$. Similarly, $x$ does not belong to $C'_2$, for all choices of $C'_2$. This implies that $x\in C'_1$ for all choices of $C'_1$ and $a\in C'_2$ for all choices of $C'_2$.
\\ In the sequel we will use the fact that $G_1$ and $G_2$ are interchangeable, as are $a$ and $x$, $G'_1$ and $G'_2$.
\\\\ {\bf Case 1} First suppose that $G_2$ fulfils $(b)$ with respect to the free neighbour $y$ of $x$. Set $\overline{G}_2=G_2\setminus\{xy\}$, let $\overline{H}$ be the connected component of $y$ in $\overline{G}_2$ and set $\overline{K}=\overline{G}_2\setminus\overline{H}$ (see Figure \ref{case3.4.5}).
Consider the graph $G_1^{\ast} = G'_1 \cup \overline{K}$.

\begin{center}
\begin{figure}[ht!]
\begin{pspicture}(-1.5,-1.7)(4,1.6)
\psline(0,0)(-0.87,0.5)
\psline(-0.87,0.5)(-0.87,-0.5)
\psline(-0.87,-0.5)(0,0)
\psline(0,0)(1,0)
\psline(0,0)(0,1)
\psline[linecolor=gray](1,0)(2,0)
\psline[linewidth=1pt,linestyle=dashed](1,0)(1.71,-0.71)
\psline[linestyle=dashed](1.71,-0.71)(1.97,-1.67)
\psline[linestyle=dashed](1.97,-1.67)(2.67,-0.97)
\psline[linestyle=dashed](2.67,-0.97)(1.71,-0.71)
\psline[linewidth=1.5pt,linestyle=dotted](2,0)(3,0)
\psline[linewidth=1.5pt,linestyle=dotted](3,0)(4,0)
\psline[linewidth=1.5pt,linestyle=dotted](2,0)(2.87,0.5)
\psline[linewidth=1.5pt,linestyle=dotted](2.87,0.5)(3.73,1)
\psdots(0,0)
\psdots[dotsize=3pt 3,dotstyle=Bo](1,0)
\psdots[dotsize=3pt 3,dotstyle=Bo](2,0)
\psdots[dotsize=3pt 3,dotstyle=Bo](-0.87,-0.5)
\psdots[dotsize=3pt 3,dotstyle=Bo](-0.87,0.5)
\psdots[dotsize=3pt 3,dotstyle=Bo](0,1)
\psdots(1.71,-0.71)
\psdots[dotsize=3pt 3,dotstyle=Bo](2.67,-0.97)
\psdots[dotsize=3pt 3,dotstyle=Bo](1.97,-1.67)
\psdots(3,0)
\psdots[dotsize=3pt 3,dotstyle=Bo](4,0)
\psdots(2.87,0.5)
\psdots[dotsize=3pt 3,dotstyle=Bo](3.73,1)
\uput[290](0,0){$a$}
\uput[250](1,0){$x$}
\uput[290](2,0){$y$}
\uput[90](0,1){$w$}
\uput[120](-0.87,0.5){$v_1$}
\rput(-1.5,0){$G'_1$}
\rput(2,1){$G_2$}
\rput(4,0.5){$\overline{H}$}
\rput(1.3,-1.1){$\overline{K}$}
\end{pspicture}
\caption{} \label{case3.4.5}
\end{figure}
\end{center}
\noindent {\small\textit{In Figure \ref{case3.4.5} the edges of $G'_1$ are thick lines, the edges of $\overline{H}$ are dotted lines, the edges of $\overline{K}$ are dashed lines and the only edge of $G_2$ that does not belong to $\overline{H}$ and $\overline{K}$ is $xy$ and is a thin grey line.}}
\\\\
Let $\overline{D}$ be a maximum minimal vertex cover of $\overline{K}$.  In $G_1^\ast$, the subgraphs $G'_1$ and $\overline{K}$ have only the vertex $x$ in common, and, moreover, $x$ belongs to all maximum minimal vertex covers of $G'_1$ and $\overline{K}$. Hence, by Lemma \ref{lemma1.1},  $C_1^\ast=C'_1\cup\overline{D}$ is a maximum minimal vertex cover of $G_1^\ast$ and $x$ belongs to all maximum minimal vertex covers of $G_1^\ast$.
But then, according to Lemma \ref{lemma1}, $C_1^\ast$ is a maximum minimal vertex cover of $G_1^\ast\cup\{xy\}$, as well. Hence, if $\overline{E}$ is a maximum minimal vertex cover of $\overline{H}$, $C_1^\ast\cup\overline{E}$ is a minimal vertex cover of $G=(G_1^\ast\cup\{xy\})\cup\overline{H}$. Thus, if $\overline{H}$ is not empty,  the claim follows by induction applied to $G_1^\ast\cup\{xy\}$ and $\overline{H}$.\\
Now suppose that $\overline{H}$ is empty. In this case $xy$ is a terminal edge of $G_2$. By Lemma \ref{lemma1}, $\overline{D}$ is a maximum minimal vertex cover of $\overline{K}\cup\{xy\}$, which, in this case, is the whole graph $G_2$. Since $x\in \overline{D}$, $\overline{D}$ is also a minimal vertex cover of $G'_2$. Consequently, $C_1 \cup \overline{D}$  is a minimal vertex cover of $G = G_1 \cup G'_2$, whence
\[
\bight I(G) \geq |C_1| + |\overline{D}|.
\]
If $G_1$ fulfils condition $(b)$, with respect to some free neighbour $w$ of $a$, then certainly $w$ is not a terminal vertex, because we are assuming that the edge $ax$ is not whiskered at both endpoints. Hence, after exchanging $G_1$ and $G_2$, we are taken back to the case in which $\overline{H}$ is not empty. So suppose that  $G_1$ fulfils condition $(a)$, $(c)$, $(d)$ or $(e)$. In view of Case 2 below, we only have to consider the first three cases. First suppose that $G_1$ fulfils $(a)$, with respect to a redundant free neighbour $w$ of $a$. Set $\overline{G}_1 = G_1 \setminus \{aw\}$ (see Figure \ref{case3.4.6}). Then $\bight I(\overline{G}_1) = \bight I(G_1) - 1$. Moreover, let $\overline{K}' = \overline{K} \cup \{ax\}$. Then, by Lemma \ref{lemma1}, $\overline{D}$ is a maximum minimal vertex cover of $\overline{K}'$. Finally, set	
\[
\overline{G} = \overline{G}_1 \cup \overline{K}' = G \setminus \{xy, aw\}.
\]

\begin{center}
\begin{figure}[ht!]
\begin{pspicture}(-1.5,-1.7)(2.7,0.7)
\psline(0,0)(-0.87,0.5)
\psline(-0.87,0.5)(-0.87,-0.5)
\psline(-0.87,-0.5)(0,0)
\psline(0,0)(1,0)
\psline[linestyle=dashed](1,0)(1.71,-0.71)
\psline[linestyle=dashed](1.71,-0.71)(1.97,-1.67)
\psline[linestyle=dashed](1.97,-1.67)(2.67,-0.97)
\psline[linestyle=dashed](2.67,-0.97)(1.71,-0.71)
\psdots(0,0)
\psdots[dotsize=3pt 3,dotstyle=Bo](-0.87,-0.5)
\psdots[dotsize=3pt 3,dotstyle=Bo](-0.87,0.5)
\psdots[dotsize=3pt 3,dotstyle=Bo](1,0)
\psdots(1.71,-0.71)
\psdots[dotsize=3pt 3,dotstyle=Bo](2.67,-0.97)
\psdots[dotsize=3pt 3,dotstyle=Bo](1.97,-1.67)
\uput[290](0,0){$a$}
\uput[250](1,0){$x$}
\uput[120](-0.87,0.5){$v_1$}
\rput(-1.5,0){$\overline{G}_1$}
\rput(1.5,0.5){$\overline{G}$}
\rput(1.3,-1.1){$\overline{K}$}
\end{pspicture}
\caption{}\label{case3.4.6}
\end{figure}
\end{center}
\noindent {\small\textit{In Figure $\ref{case3.4.6}$ the edges of $\overline{G}_1$ are thick lines.}}
\\\\ Then $I(G) = I(\overline{G}) + (xy,aw)$, whence $I(G) = \sqrt{I(\overline{G}) + (xy+aw)}$, because $ax \in I(\overline{G})$ and $x^2y^2=xy(xy+aw)-axyw$, so that
\[
\ara I(G) \leq \ara I(\overline{G}) + 1.
\]
Now, induction applies to $\overline{G}_1$ and $\overline{K}'$, ($\overline{G}_1$ may be empty) so that
\begin{eqnarray*}
\bight I(G) - 1 &\geq& |C_1| - 1 + |\overline{D}|\\
&=& \bight I(\overline{G}_1) + \bight I(\overline{K}')\\
&\geq& \ara I(\overline{G}_1)-n_1 + \ara I(\overline{K}') - n_2\\
&\geq& \ara I(\overline{G}) - n\\
&\geq& \ara I(G) - n - 1,
\end{eqnarray*}
whence the desired inequality for $G$.
\\ \noindent Now suppose that $(c)$ holds for $G_1$ with respect to a redundant non-free neighbour $v_1$ of $a$. Set $\overline{G}_1 = G_1 \setminus \{av_1\}$. Then $\bight I(\overline{G}_1)=\bight I(G_1)-1$ and the number of cycles of $\overline{G}_1$ is $n_1-1$. Then the same computation as above yields the desired inequality.
\\ Finally, suppose that $(d)$ holds for $G_1$ with respect to $v_1$. Let $v_2$ be the other non-free redundant neighbour of $x$,  set $\widetilde{G}_1=G_1\setminus\{av_1, av_2\}$. Then  $\bight I(\widetilde{G}_1)=\bight I(G_1)-1$ and the number of cycles of $\widetilde{G}_1$ is $n_1-1$. Let
\[
\widetilde{G} = \widetilde{G}_1 \cup \overline{K}' = G \setminus \{xy, av_1, av_2\}.
\]
		
\noindent Then $I(G) = I(\widetilde{G}) + (xy,av_1, av_2)$, whence $I(G) = \sqrt{I(\widetilde{G}) + (xy+av_1, av_2)}$,  and thus
\[
\ara I(G) \leq \ara I(\widetilde{G}) + 2.
\]
Now, induction applies to $\widetilde{G}_1$ and $\overline{K}'$,  so that
\begin{eqnarray*}
\bight I(G) - 1 &\geq& |C_1| - 1 + |\overline{D}|\\
&=& \bight I(\widetilde{G}_1) + \bight I(\overline{K}')\\
&\geq& \ara I(\widetilde{G}_1)-n_1 +1 + \ara I(\overline{K}') - n_2\\
&\geq& \ara I(\widetilde{G}) - n+1\\
&\geq& \ara I(G) - n - 1,
\end{eqnarray*}
which, again, yields the desired inequality for $G$.		
\\[3mm]
{\bf Case 2} Now suppose that $(e)$ applies to $G_2$, with respect to some non-free neighbour $z_1$ of $x$ in $C_2$. Let $z_2$ be the other non-free neighbour of $x$. Set $\overline{G}_2 = G_2 \setminus \{xz_1\}$. Then induction applies to $\overline{G} = G'_1 \cup \overline{G}_2$, so that $\bight I(\overline{G}) \geq \ara I(\overline{G}) - n+1$, since $n-1$ is the number of cycles contained in $\overline{G}$. On the other hand, since $I(G) = I(\overline{G}) + (xz_1)$, we also have that $\ara I(G) \leq \ara I(\overline{G}) + 1$. Let $\widetilde{H}$ be the connected component of $z_1$ in $\widetilde{G}_2 = G_2 \setminus \{xz_1, xz_2\}$ and set $\widetilde{K} = \widetilde{G}_2 \setminus \widetilde{H}$. Let $C'_1$ be a maximum minimal vertex cover of $G'_1$ (whence $x \in C'_1$), $\widetilde{D}$  a maximum minimal vertex cover of $\widetilde{K}$, and $\widetilde{E}$  a maximum minimal vertex cover of $\widetilde{H}$. Since $x$ is the only common vertex of $G'_1$ and $\widetilde{K}$, and $x$ belongs to all maximum minimal vertex covers of $G'_1$ and of $\widetilde{K}$, by Lemma \ref{lemma1.1} we have that $C'_1 \cup \widetilde{D}$ is a maximum minimal vertex cover of $G'_1\cup\widetilde{K}$ and $x$ belongs to all maximum minimal vertex covers of this graph. On the other hand, $G'_1\cup\widetilde{K}$ is vertex-disjoint from $\widetilde{H}$,  and $\overline{G}$ is obtained by connecting  $G_1'\cup\widetilde{K}$ and $\widetilde{H}$ through  the edge $xz_2$. Hence, by Lemma \ref{lemma1.2} $(ii)$, $\overline{C}=C'_1 \cup \widetilde{D}\cup\widetilde{E}$ is a  maximum minimal vertex cover of $\overline{G}$.  Moreover, it is a minimal vertex cover of $G$. Therefore $\bight I(G) \geq \bight I(\overline{G})$.  Finally, induction applies to $\overline{G}$. Summing up, we have
\[
\bight I(G) \geq \bight I(\overline{G}) \geq \ara I(\overline{G}) - n+1 \geq  \ara I(G) - n.
\]
\\[1mm]
The cases where $G_1$ fulfils $(b)$ or  $(e)$ can be treated in the same way as in Cases 1 and 2 above.
\\[3mm]
{\bf Case 3} Finally, suppose that each of $G_1$ and $G_2$ fulfils $(a)$ or $(c)$ or $(d)$. \\
A subset $W$ of $G_1$ will be called a \textit{neighbour set} of $a$ if $W=\{aw\}$ for some neighbour $w$ of $a$ or $W=\{av_1, av_2\}$ where $v_1, v_2$ are non-free neighbours of $a$.
Set
\[
h_1 = \max \left\{ h\ \Bigg|
\begin{tabular}{ccc}
$\exists W_1,\dots, W_h$ neighbour sets of $a$ in $G_1$ such that,\\
for all $\ell=1,\dots,h$, $\bight I(G_1 \setminus\bigcup_{i=1}^\ell W_i) = \bight I(G_1) - \ell$
\end{tabular}
\right\}.
\]
Consider
\[
\widehat{G}_1 = G_1 \setminus\bigcup_{i=1}^{h_1} W_i.
\]
Then $h_1\geq 1$. In view of Lemma \ref{lemma3}, the maximality of $h_1$ implies that either $\widehat{G}_1$ is empty or one of the following conditions holds.
\begin{itemize}
\item[$(i)$] All maximum minimal vertex covers of $\widehat{G}_1$ contain $a$ (and are thus maximum minimal vertex covers of $\widehat{G}'_1 = \widehat{G}_1 \cup \{ax\}$, as well).
\vspace{2mm}
\\ In the remaining cases, there is a maximum minimal vertex cover $\widehat{C}_1$ of $\widehat{G}_1$ such that $a \notin \widehat{C}_1$. In view of Lemma \ref{lemma1}, this implies that $\bight I(\widehat{G}_1')=\bight I(\widehat{G}_1)+1$.
\vspace{2mm}
\item[$(ii)$] For some maximum minimal vertex cover $\widehat{C}_1$ of $\widehat{G}_1$ such that $a \notin \widehat{C}_1$, no neighbour of $a$ is redundant. In this case $\widehat{C}_1 \cup \{a\}$ is a maximum minimal vertex cover of $\widehat{G}'_1$.
\\[2mm] In the remaining cases, for all maximum minimal vertex covers $\widehat{C}_1$ of $\widehat{G}_1$ such that $a \notin \widehat{C}_1$, there is some redundant neighbour of $a$ with respect to $\widehat{C}_1$. This implies that $a$ does not belong to any maximum minimal vertex cover of $\widehat{G}'_1$.
\vspace{2mm}
\item[$(iii)$] For all maximum minimal vertex covers $\widehat{C}_1$ of $\widehat{G}_1$ such that $a \notin \widehat{C}_1$, in $\widehat{C}_1$ there is some redundant free neighbour of $a$. Then, by Lemma \ref{lemma3}, there is  such a neighbour $w$ for which the following holds. Set $\widehat{G}_1^{\bullet} = \widehat{G}_1 \setminus \{aw\}$, call $\overline{H}$ the connected component of $w$ in $\widehat{G}_1^{\bullet}$, and set $\overline{K} = \widehat{G}_1^{\bullet} \setminus \overline{H}$. Then $a$ belongs to all maximum minimal vertex covers of $\overline{K}$.
\item[$(iv)$] For some maximum minimal vertex cover $\widehat{C}_1$ of $\widehat{G}_1$ such that $a \notin \widehat{C}_1$, there are no redundant free neighbours of $a$ with respect to $\widehat{C}_1$, but there is a redundant non-free neighbour $v_1$ such that the following holds. Call $v_2$ the other non-free neighbour of $a$, and set $\widehat{G}_1^\vee = \widehat{G}_1 \setminus \{av_1,av_2\}$, call $\widetilde{H}$ the connected component of $v_1$ in $\widehat{G}_1^\vee$, and set $\widetilde{K} = \widehat{G}_1^\vee \setminus \widetilde{H}$. Then $a$ belongs to all maximum minimal vertex covers of $\widetilde{K}$.
\end{itemize}
\vspace{2mm}
Define $h_2$ for $G_2$, in the same way as $h_1$ for $G_1$, i.e., set
\[
h_2 = \max \left\{ h\ \Bigg|
\begin{tabular}{ccc}
$\exists U_1,\dots,U_h$ neighbour sets of $x$ in $G_2$ such that,\\
for all $\ell=1,\dots,h$, $\bight I(G_2 \setminus \bigcup_{i=1}^\ell U_i) = \bight I(G_2) - \ell$
\end{tabular}
\right\}.
\]
Suppose that $h_1\leq h_2$, and then set
\[
\widehat{G}_2 = G_2 \setminus \bigcup_{i=1}^{h_1} U_i.
\]
In the sequel we will admit that $\widehat{G}_2$ may be empty.
\\ We have
\[
I(G) = I(\widehat{G}'_1 \cup \widehat{G}_2) + I(W_1\cup\cdots\cup W_{h_1}\cup U_1\cup\cdots\cup U_{h_1}).
\]
Note that at most one of the sets $W_i$ and at most one of the sets $U_i$ consists of two elements. For all $i=1,\dots, h_1$, let $aw_i\in W_i$ and $xu_i\in U_i$. If some $W_i$ contains another element, call it $\alpha$, otherwise set $\alpha=0$.  If some $U_i$ contains another element, call it $\beta$, otherwise set $\beta=0$.
Then
\[
I(G) = I(\widehat{G}'_1 \cup \widehat{G}_2) + (aw_1,\dots,aw_{h_1}, xu_1,\dots, xu_{h_1}, \alpha, \beta),
\]
so that
\[
I(G) = \sqrt{I(\widehat{G}'_1 \cup \widehat{G}_2) + (aw_1+xu_1,\dots,aw_{h_1}+xu_{h_1}, \alpha+\beta)},
\]
because $ax \in I(\widehat{G}'_1\cup\widehat{G}_2)$. Hence
\[
\ara I(G) \leq \ara I(\widehat{G}'_1 \cup \widehat{G}_2) + \widehat{h}_1 \leq \ara I(\widehat{G}'_1) + \ara I(\widehat{G}_2) + \widehat{h}_1,
\]
where $\widehat{h}_1=h_1$ or $\widehat{h}_1=h_1+1$. In the latter case, at least one cycle is lost when passing from $G$ to $\widehat{G}'_1\cup\widehat{G}_2$, and consequently, if $\widehat{n}$ is the number of cycles of this graph, we have $\widehat{n}\leq n-1$, whereas, in general, $\widehat{n}\leq n$. Thus we always have $\widehat{n}+\widehat{h}_1\leq n+h_1$, whence
\[
h_1-\widehat{n}\geq \widehat{h}_1-n.
\]
Moreover,
\[
\bight I(G_2)=\bight I(\widehat{G}_2)+h_1.
\]
Let $\widehat{n}_1$ be the number of cycles contained in $\widehat{G}'_1$, and $\widehat{n}_2$ be the number of cycles contained in $\widehat{G}_2$, so that $\widehat{n}=\widehat{n}_1+\widehat{n}_2$.  Induction applies to $\widehat{G}'_1$ and $\widehat{G}_2$.  Hence
\[
\bight I(\widehat{G}'_1) \geq \ara I(\widehat{G}'_1) - \widehat{n}_1,
\]
and
\[
\bight I(\widehat{G}_2) \geq \ara I(\widehat{G}_2) - \widehat{n}_2.
\]
In cases $(i)$ and $(ii)$, let $\widehat{C}_1$ be a maximum minimal vertex cover of $\widehat{G}_1$. Then, in case $(i)$, $\widehat{C}'_1=\widehat{C}_1$ is also a maximum minimal vertex cover of  $\widehat{G}'_1$.   In case $(ii)$, $\widehat{C}'_1=\widehat{C}_1\cup\{a\}$ is a maximum minimal vertex cover of $\widehat{G}'_1$. In both cases, $\widehat{C}'_1 \cup C_2$ is a minimal vertex cover of $G$. Moreover, $\widehat{C}'_1$ and $C_2$ are disjoint. Consequently,
\begin{eqnarray*}
\bight I(G) \geq |\widehat{C}'_1| + |C_2| &=& \bight I(\widehat{G}'_1) + \bight I(G_2)\\
&=& \bight I(\widehat{G}'_1) + \bight I(\widehat{G}_2) + h_1\\
&\geq& \ara I(\widehat{G}'_1) - \widehat{n}_1 + \ara I(\widehat{G}_2) - \widehat{n}_2 + h_1\\
&=& \ara I(\widehat{G}'_1)+ \ara I(\widehat{G}_2)+h_1 - \widehat{n} \\
&\geq& \ara I(\widehat{G}'_1) +\ara I(\widehat{G}_2) + \widehat{h}_1-n\\
&\geq& \ara I(G) - n.
\end{eqnarray*}	
Before we examine the remaining cases, one remark is needed. First suppose that all maximum minimal vertex covers $\widehat{C}_2$ of $\widehat{G}_2$ contain $x$. Then we are taken back to case $(i)$ with $\widehat{G}'_1$ replaced by $\widehat{G}'_2=\widehat{G}_2\cup\{ax\}$  and $\widehat{G}_2$ replaced by $\widehat{G}_1$ (note that the maximality of $h_1$ is irrelevant in this part of the argumentation). Then suppose that for some maximum minimal vertex cover $\widehat{C}_2$ of $\widehat{G}_2$ such that $x \notin \widehat{C}_2$, there is no redundant neighbour of $x$ with respect to $\widehat{C}_2$. Then, for all neighbour sets $U$ of $x$ in $\widehat{G}_2$, $\widehat{C}_2$ is also a minimal vertex cover of $\widehat{G}_2 \setminus U$. Thus the elimination of the neighbour set $U$ does not cause the big height to drop. This implies that $h_2 = h_1$. Thus we are taken back to case $(ii)$, with the roles of $\widehat{G}_1$ and $\widehat{G}_2$ exchanged. Hence, in the sequel, we may assume that, for all maximum minimal vertex covers $\widehat{C}_2$ of $\widehat{G}_2$ such that $x \notin \widehat{C}_2$ (and such covers exist), there is some redundant neighbour of $x$ with respect to $\widehat{C}_2$. Hence $x$ does not belong to any maximum minimal vertex cover of $\widehat{G}'_2$. Recall that the same is true for $G'_2$.
\\[1mm]
In case $(iii)$, let $\widehat{C}_1^\bullet$ be a maximum minimal vertex cover of $\widehat{G}_1^\bullet$; $\widehat{G}_1$ fulfils condition $(b)$ with respect to the neighbour $w$ of $a$, so that, by Corollary \ref{remark} $(i)$, $a$ belongs to all maximum minimal vertex covers of $\widehat{G}_1^\bullet$, and, in particular, $a \in \widehat{C}_1^\bullet$. Moreover, $\widehat{C}_1^\bullet$ is a maximum minimal vertex cover of $\widehat{G}_1$. Furthermore, in $\widehat{G}'_1\cup \widehat{G}_2$, the subgraphs  $\widehat{G}_1$ and $\widehat{G}_2$ are connected through the edge $ax$. Recall that $a$ does not belong to any maximum minimal vertex cover of $\widehat{G}'_1$.  Thus, in view of Lemma \ref{lemma1.2} $(i)$, $\widehat{C}_1^\bullet\cup \widehat{C}_2$ is a maximum minimal vertex cover of $\widehat{G}'_1\cup \widehat{G}_2$.
Moreover, $\widehat{C}_1^\bullet\cup C_2$  is a minimal vertex cover of  $G$. Hence, by induction,
\begin{eqnarray*}
\bight I(G) \geq |\widehat{C}_1^\bullet\cup C_2|&=& |\widehat{C}_1^\bullet|+|C_2|= |\widehat{C}_1^\bullet|+|\widehat{C}_2|+h_1\\
&=& \bight I(\widehat{G}'_1\cup\widehat{G}_2) + h_1\\
&\geq& \ara I(\widehat{G}'_1\cup\widehat{G}_2) - \widehat{n} + h_1\\
&\geq& \ara I(G) - n.
\end{eqnarray*}
In case $(iv)$, let $\widehat{C}_1^\vee$ be a maximum minimal vertex cover of $\widehat{G}_1^\vee$. Then, by Corollary \ref{remark} $(ii)$, $a \in \widehat{C}_1^\vee$, and $\widehat{C}_1^\vee$ is also a maximum minimal vertex cover of $\widehat{G}_1^\bullet = \widehat{G}_1 \setminus \{av_1\}$.
\\  If $a$ does not belong to any maximum minimal vertex cover of $\widehat{G}_1^\bullet\cup\{ax\}=\widehat{G}_1'\setminus\{av_1\}$, then by Lemma \ref{lemma1.2} $(i)$, $\widehat{C}_1^\vee \cup \widehat{C}_2$ is a maximum minimal vertex cover of $(\widehat{G}'_1 \setminus\{av_1\})\cup \widehat{G}_2=\widehat{G}_1^\bullet\cup\{ax\}\cup\widehat{G}_2$, since this graph  is obtained by connecting $\widehat{G}_1^\bullet$ and $\widehat{G}_2$ through the edge $ax$. Furthermore, $\widehat{C}_1^\vee \cup C_2$ is a minimal vertex cover of $G$.
\\ Now, induction applies to $(\widehat{G}'_1 \setminus \{av_1\}) \cup \widehat{G}_2$, and the number of cycles contained in this graph is  $\widehat{n}-1$. Hence
\begin{eqnarray*}
\bight I(G) \geq |\widehat{C}_1^\vee \cup C_2| &=& |\widehat{C}_1^\vee \cup \widehat{C}_2| +h_1\\
&=& \bight I((\widehat{G}'_1 \setminus \{av_1\}) \cup \widehat{G}_2) + h_1\\
&\geq& \ara I((\widehat{G}'_1 \setminus \{av_1\}) \cup \widehat{G}_2) - \widehat{n} + 1 + h_1\\
&\geq& \ara I(\widehat{G}'_1 \cup \widehat{G}_2) - 1 - \widehat{n} + 1 + h_1\\
&\geq& \ara I(G) - n.
\end{eqnarray*}
If $a$ belongs to some maximum minimal vertex cover $\widehat{C}_1'$ of $\widehat{G}_1'\setminus\{av_1\}$, then $\widehat{C}_1'\cup C_2$ is a minimal vertex cover of $G$. Moreover, $\widehat{G}'_1\setminus\{av_1\}$ contains $n_1-1=\widehat{n}_1-1$ cycles. Hence, by induction, we have
\begin{eqnarray*}
\bight I(G) &\geq& |\widehat{C}'_1|+|C_2|=|\widehat{C}'_1|+|\widehat{C}_2|+h_1\\
&=& \bight I(\widehat{G}'_1\setminus\{av_1\})+ \bight I(\widehat{G}_2) + h_1\\
&\geq& \ara I(\widehat{G}'_1\setminus\{av_1\})- \widehat{n}_1+1 + \ara I(\widehat{G}_2)  -\widehat{n}_2 + h_1\\
&\geq& \ara I(\widehat{G}'_1)+ \ara I(\widehat{G}_2)  -\widehat{n} + h_1\\
&\geq& \ara I(G) - n.
\end{eqnarray*}
\vspace{2mm}
\\ If $h_1 > h_2$, it suffices to apply the above arguments after exchanging the roles of $G_1$ and $G_2$. This completes the proof.
\end{proof}

\section{Final Remarks}

According to Kuratowski's Theorem (see, e.g.,  \cite{Ha}, Theorem 11.13)  the graphs whose cycles are pairwise vertex-disjoint are all planar. Moreover, the number $n$ of cycles of a graph $G$ fulfilling the assumption of Theorem \ref{main} coincides with the so-called \textit{cycle rank} of $G$, which, according to \cite{Ha}, Corollary 4.5 (b), is equal to $e-\vert V(G)\vert +k$, where $e$ is the number of edges and $k$ is the number of connected components.
\newline\vskip.1truecm\noindent
Theorem \ref{main} implies that, whenever $G$ is acyclic, $\ara I(G)\leq\bight I(G)$. In view of \eqref{0}, it follows that, in this case, $\ara I(G)=\pd R/I(G)=\bight I(G)$. This result was conjectured in \cite{B}, where it was proven for a special class of acyclic graphs. The general case was settled by Kimura and Terai in \cite{KT13}. Equality can also hold for a graph containing an arbitrary number of cycles: an infinite class of such examples is provided by the fully whiskered graphs on graphs with pairwise disjoint cycles. Therefore the inequality given in Theorem \ref{main} is strict in general. On the other hand, the bound given there is sharp, as is shown by the following example, which is taken from \cite{M13c}.
\\ Consider the graphs $G_1$ and $G_2$ depicted in Figure \ref{F.example}.

\begin{figure}[ht!]\centering
\psset{unit=1.4cm}
\begin{pspicture}(-0.8,-0.6)(5,1.5)
\psline(0,0)(1,0)
\psline(1,0)(1,1)
\psline(1,1)(0,1)
\psline(0,1)(0,0)
\psline(3,1)(3,0)
\psline(1,1)(4,1)
\psdots(0,0)
\psdots[dotsize=3pt 3,dotstyle=Bo](1,0)
\psdots(1,1)
\psdots[dotsize=3pt 3,dotstyle=Bo](0,1)
\psdots[dotsize=3pt 3,dotstyle=Bo](2,1)
\psdots(3,1)
\psdots[dotsize=3pt 3,dotstyle=Bo](3,0)
\psdots[dotsize=3pt 3,dotstyle=Bo](4,1)

\rput(-0.7,0.5){$G_1$}
\uput[90](1,1){$x^{(1)}_1$}
\uput[90](0,1){$x^{(1)}_2$}
\uput[270](0,0){$x^{(1)}_3$}
\uput[270](1,0){$x^{(1)}_4$}
\uput[90](2,1){$y^{(1)}_1$}
\uput[90](3,1){$y^{(1)}_2$}
\uput[90](4,1){$y^{(1)}_3$}
\uput[270](3,0){$y^{(1)}_4$}
\end{pspicture}
\psset{unit=1.1cm}
\begin{pspicture}(-0.5,-0.5)(10,1.6)
\psline(0,0)(1,0)
\psline(1,0)(1,1)
\psline(1,1)(0,1)
\psline(0,1)(0,0)
\psline(6,1)(5,1)
\psline(5,1)(5,0)
\psline(5,0)(6,0)
\psline(6,0)(6,1)
\psline(3,1)(3,0)
\psline(8,1)(8,0)
\psline(6,0)(9,0)
\psline(1,1)(5,1)
\psdots(0,0)
\psdots[dotsize=3pt 3,dotstyle=Bo](1,0)
\psdots(1,1)
\psdots[dotsize=3pt 3,dotstyle=Bo](0,1)
\psdots[dotsize=3pt 3,dotstyle=Bo](2,1)
\psdots(3,1)
\psdots[dotsize=3pt 3,dotstyle=Bo](3,0)
\psdots[dotsize=3pt 3,dotstyle=Bo](4,1)
\psdots(5,1)
\psdots[dotsize=3pt 3,dotstyle=Bo](6,1)
\psdots[dotsize=3pt 3,dotstyle=Bo](5,0)
\psdots(6,0)
\psdots[dotsize=3pt 3,dotstyle=Bo](7,0)
\psdots(8,0)
\psdots[dotsize=3pt 3,dotstyle=Bo](9,0)
\psdots[dotsize=3pt 3,dotstyle=Bo](8,1)

\rput(-0.7,0.5){$G_2$}
\uput[90](1,1){$x^{(1)}_1$}
\uput[90](0,1){$x^{(1)}_2$}
\uput[270](0,0){$x^{(1)}_3$}
\uput[270](1,0){$x^{(1)}_4$}
\uput[90](2,1){$y^{(1)}_1$}
\uput[90](3,1){$y^{(1)}_2$}
\uput[90](4,1){$y^{(1)}_3$}
\uput[270](3,0){$y^{(1)}_4$}

\uput[270](6,0){$x^{(2)}_1$}
\uput[270](5,0){$x^{(2)}_2$}
\uput[90](5,1){$x^{(2)}_3$}
\uput[90](6,1){$x^{(2)}_4$}
\uput[270](7,0){$y^{(2)}_1$}
\uput[270](8,0){$y^{(2)}_2$}
\uput[270](9,0){$y^{(2)}_3$}
\uput[90](8,1){$y^{(2)}_4$}
\end{pspicture}
\caption{}\label{F.example}
\end{figure}
\noindent {\small\textit{In Figure $\ref{F.example}$ the empty dots form maximum minimal vertex covers of $G_1$ and $G_2$.}}
\\\\
 We have $\bight I(G_1) = 5$ and $\bight I(G_2) = 10$. The projective dimensions of the corresponding quotient rings (in characteristic zero) can be computed by the software packages \texttt{CoCoA} \cite{CoCoA} or \texttt{Macaulay2} \cite{Mac2} and provide a lower bound for the arithmetical rank, namely $6\leq\ara I(G_1)$, and $12\leq\ara I(G_2)$. On the other hand, we also have the opposite inequalities. In fact, the polynomials\\

\begin{tabular}{p{6.3cm}p{6.3cm}}
$\begin{aligned}
q^{(1)}_0 &= x^{(1)}_1 x^{(1)}_2, \\
q^{(1)}_1 &= x^{(1)}_1 x^{(1)}_4 + x^{(1)}_2 x^{(1)}_3, \\
q^{(1)}_2 &= x^{(1)}_1 y^{(1)}_1 + x^{(1)}_3 x^{(1)}_4,
\end{aligned}$
&
$\begin{aligned}
q^{(1)}_3 &= y^{(1)}_2 y^{(1)}_3, \\
q^{(1)}_4 &= y^{(1)}_1 y^{(1)}_2, \\
q^{(1)}_5 &= y^{(1)}_2 y^{(1)}_4
\end{aligned}$
\end{tabular}
\vspace*{1mm}\\ \noindent generate an ideal whose radical is $I(G_1)$,  and the polynomials\\[2mm]
\begin{tabular}{p{6.3cm}p{6.3cm}}
$\begin{aligned}
q^{(1)}_0 &= x^{(1)}_1 x^{(1)}_2, \\
q^{(1)}_1 &= x^{(1)}_1 x^{(1)}_4 + x^{(1)}_2 x^{(1)}_3, \\
q^{(1)}_2 &= x^{(1)}_1 y^{(1)}_1 + x^{(1)}_3 x^{(1)}_4, \\
q^{(1)}_3 &= y^{(1)}_2 y^{(1)}_3, \\
q^{(1)}_4 &= y^{(1)}_1 y^{(1)}_2 + y^{(1)}_3 x^{(2)}_3, \\
q^{(1)}_5 &= y^{(1)}_2 y^{(1)}_4
\end{aligned}$
&
$\begin{aligned}
q^{(2)}_0 &= x^{(2)}_1 x^{(2)}_2, \\
q^{(2)}_1 &= x^{(2)}_1 x^{(2)}_4 + x^{(2)}_2 x^{(2)}_3, \\
q^{(2)}_2 &= x^{(2)}_1 y^{(2)}_1 + x^{(2)}_3 x^{(2)}_4, \\
q^{(2)}_3 &= y^{(2)}_2 y^{(2)}_3, \\
q^{(2)}_4 &= y^{(2)}_1 y^{(2)}_2, \\
q^{(2)}_5 &= y^{(2)}_2 y^{(2)}_4
\end{aligned}$
\end{tabular}
\\[2mm] generate an ideal whose radical is $I(G_2)$.
\\ Hence $6=\ara I(G_1)=\bight I(G_1) + 1$ and $12=\ara I(G_2)=\bight I(G_2) + 2$.

\bibliographystyle{plain}	                  

\end{document}